\documentclass[a4paper, 12pt]{amsart}

\usepackage{amssymb}
\usepackage{amsmath}
\usepackage{amscd}
\usepackage{amsthm}
\usepackage[centertags]{amsmath}
\usepackage{amsfonts}
\usepackage{newlfont}
\usepackage[all]{xy}
\usepackage{graphicx}
\usepackage{amsfonts, amssymb}
\usepackage[usenames]{color}
\usepackage{mathrsfs}
\usepackage{latexsym}

\newtheorem{thm}{Theorem}[section]
\newtheorem{cor}[thm]{Corollary}
\newtheorem{lem}[thm]{Lemma}
\newtheorem{prop}[thm]{Proposition}

\theoremstyle{definition}

\newtheorem{rem}[thm]{Remark}

\newcommand{\kkk}{\mathbf k}
\newcommand{\ccc}{\mathbf c}
\newcommand{\MMM}{\mathbf M}

\newcommand{\zC}{\mathbb C}
\newcommand{\zQ}{\mathbb Q}
\newcommand{\zR}{\mathbb R}

\newcommand{\zN}{\mathbb N}
\newcommand{\zK}{\mathbb K}

\begin{document}
\baselineskip=.75cm

\title{Non-linear Plank Problems and polynomial inequalities}

\author[D. Carando]{Daniel Carando}
\address{Departamento de Matem\'{a}tica - Pab I,
Facultad de Cs. Exactas y Natu\-ra\-les, Universidad de Buenos Aires,
(1428) Buenos Aires, Argentina and IMAS-CONICET}
\email{dcarando@dm.uba.ar}

\author[D. Pinasco]{Dami\'an Pinasco}
\address{Departamento de Matem\'{a}ticas y Estad\'{i}stica, Universidad Torcuato Di Tella, Av. Figueroa Alcorta 7350, (1428) Buenos Aires, Argentina and CONICET}
\email{dpinasco@utdt.edu}

\author[J. T. Rodr\'{i}guez]{Jorge Tom\'as Rodr\'{i}guez}
\address{Departamento de Matem\'{a}tica, Facultad de Cs. Exactas, Universidad Nacional del Centro de la Provincia de Buenos Aires, (7000) Tandil, Argentina and NUCOMPA-UNICEN}
\email{jtrodrig@dm.uba.ar}

\begin{abstract}
We study lower bounds for the norm of the product of polynomials and their applications to the so called \emph{plank problem.} We are particularly interested in polynomials on finite dimensional Banach spaces, in which case our results improve previous works when the number of polynomials is large.
\end{abstract}

\maketitle


\section*{Introduction}

The problem of finding lower bounds for the product of polynomials has been studied in several situations, considering a wide variety underlying spaces and norms. On a Banach space $X$, our study focuses on finding the best constant $M$ such that, for any set of continuous scalar polynomials $P_1, \ldots, P_n$ over $X$, of some prescribed degrees, the following inequality holds
\begin{equation}\label{problema}
 \Vert P_1 \Vert \cdots \Vert P_n \Vert \leq M \Vert P_1 \cdots P_n\Vert .
\end{equation}
In \cite{BST}, C.~Ben\'{i}tez, Y.~Sarantopoulos and A.~Tonge proved that, for continuous polynomials of degrees $k_1,\ldots,k_n$, inequality (\ref{problema}) holds for every complex Banach space with constant
\begin{equation}\label{eq-BST}
M=\frac{(k_1+\cdots + k_n)^{k_1+\cdots +k_n}}{k_1^{k_1} \cdots k_n^{k_n}}.
\end{equation}
The authors also showed that this is the best universal constant, since there are polynomials on $\ell_1$ for which we have equality.
For complex Hilbert spaces and homogeneous polynomials, the second named author proved in \cite{P} that the optimal constant is
\begin{equation}\label{eq-P}
M=\sqrt{\frac{(k_1+\cdots + k_n)^{k_1+\cdots +k_n}}{k_1^{k_1} \cdots k_n^{k_n}}},
\end{equation}
when the dimension of the space is at least the number of polynomials. Using a complexification argument it is easy to find a constant for real Hilbert space from \eqref{eq-P}. But, in this context, a stronger result was given in \cite{MNPP} by D.~Malicet, I.~Nourdin, G.~Peccati and G.~Poly: for real Hilbert spaces and homogeneous polynomials, (\ref{problema})~holds with constant
\begin{equation}\label{eq-MNPP}
M=\sqrt{\frac{2^{k_1+\cdots +k_n}\Gamma\left(k_1+\cdots +k_n +\frac{d}{2}\right)}{\Gamma\left(\frac{d}{2}\right)k_1! \cdots k_n!}}.
\end{equation}
For homogeneous polynomials on $L_p$ spaces or on the Schatten classes $\mathcal S_p$, with $1\leq p \leq 2$, in \cite{CPR} we showed that the optimal constant is
\begin{equation}\label{eq-CPR}
M=\sqrt[p]{\frac{(k_1+\cdots + k_n)^{k_1+\cdots +k_n}}{k_1^{k_1} \cdots k_n^{k_n}}}.
\end{equation}

Some further references on this and related problems, where other polynomial norms and different Banach spaces are considered, are \cite{A, BBEM, BR,  RS, Ro}.

In this work we aim to find better constants for finite dimensional spaces. We are able to improve some of the previous results when the number of polynomials is much larger than the dimension of the space (see Theorem~\ref{prop-cualquierespacio} and the comments following it). We also obtain specific bounds for Hilbert spaces.

\bigskip

As an application of the different versions of \eqref{problema}, we address a polynomial version of the  \emph{plank problem}. The following problem was posed  by Alfred Tarski  in the early 1930's \cite{Tar1, Tar2}.
\begin{quote} Given a convex body $K \subset \zR^d$, of minimal width 1, when $K$ is covered by $n$ parallel strips or planks with  widths $a_1, \ldots, a_n,$ is it true that $\sum_{i=1}^n a_i \ge 1$?
\end{quote}
The solution to this problem was given by T.~Bang \cite{Ban}, who also presented the following related question.
\begin{quote} When a convex body is covered by planks, is it true that the sum of the relative widths is greater than or equal to $1$?
\end{quote}
This question remains unanswered in the general case, but for centrally symmetric convex bodies the solution was given by K.~Ball in \cite{Ba1}, where he proved (slightly more than) the following.
\begin{quote} If $(\phi_j)_{j\in \zN}$ is a sequence of norm 1 linear functionals on a (real) Banach space $X$ and $(a_j)_{j \in \zN}$ is a sequence of non-negative numbers whose sum is less than 1, then there is a point $\textbf{z}$ in the unit ball of  $X$ for which $\vert \phi_j(\textbf{z})\vert \geq a_j$ for every $j \in \zN.$
\end{quote}
To realize that this is a sharp result, it is enough to consider $X=\ell_1$ and the vectors of the standard basis of its dual, $\ell_\infty$. However, when we restrict ourselves to some special Banach spaces and functionals, better constraints can be found. For example, given $\{\phi_1, \ldots, \phi_n\}$ a set of orthonormal linear functionals defined on a Hilbert space $\mathcal H,$ it is clear that for any set of real numbers $\{a_1, \ldots, a_n\}$ such that $\sum_{j=1}^n a_j^2 \le 1,$ it is possible to find a vector $\textbf{z}$ in the unit ball of  $\mathcal H$ satisfying $\vert \phi_j(\textbf{z})\vert \ge a_j$ for $j=1, \ldots, n.$ This is not  necessary true if we choose other sets of unit functionals on a real Hilbert space. For complex Hilbert spaces the situation is better, as K.~Ball showed in \cite{Ba2}:
\begin{quote}
If $(\phi_j)_{j\in \zN}$ is a sequence of norm 1 linear functionals on a complex Hilbert space $\mathcal H$ and $(a_j)_{j \in \zN}$ is a sequence of non-negative numbers satisfying $\sum_{j=1}^\infty a_j^2=1,$ then there is a unit vector $\textbf{z} \in \mathcal H$ for which $\vert \phi_j(\textbf{z})\vert \geq a_j$ for every $j \in \zN.$
\end{quote}
This result implies the following inequality: let $S_\mathcal H$ denote the unit sphere of $\mathcal H,$ then for any set of vectors $\{\omega_1, \ldots, \omega_n\} \subset S_\mathcal H$ we have
\[
\sup_{\textbf{z} \in S_\mathcal H} \vert \langle \textbf{z}, \omega_1 \rangle \cdots \langle \textbf{z}, \omega_n \rangle \vert \ge \dfrac{1}{\sqrt{n^n}}.
\]
The last inequality was proved by J.~Arias-de-Reyna \cite{A} a few years before Ball's article  using a different technique. It is related to the lower bounds for the norm of the product of polynomials mentioned above.

Using results from \cite{BST, P}, A.~Kavadjiklis and S.~G.~Kim \cite{KK} studied a plank type problem for polynomials on Banach spaces. In this article we exploit the lower bounds for the product of polynomials given in \cite{BST, CPR, P}, as well as the lower bounds we study in Section~\ref{results}, to address this kind of problems.

By a polynomial plank problem we mean to give  conditions such that, for any set of positive real numbers $a_1, \ldots, a_n,$ fulfilling them, and any set of continuous scalar polynomials $P_1,\ldots,P_n$ over a Banach space $X$, of degrees $k_1,\ldots,k_n$, there is a vector $\textbf{z}$ in the unit ball of  $X$ for which $|P_j(\textbf{z})| \ge a_j^{k_j}$ for $j=1,\ldots,n.$

\subsection*{Organization} This article is organized as follows: our lower bound for the product of polynomials, Theorem~\ref{prop-cualquierespacio} and  Propositions~\ref{prop-lowerboundhilbert} and \ref{prop-dim2}; and our plank type results, Theorems~\ref{main thm} and \ref{main thm2}, and Proposition~\ref{main thm3}, will be stated in Section~\ref{results}. In this section we also analyse these results and compare them with some previous work. The proof of Theorem~\ref{prop-cualquierespacio} and Propositions~\ref{prop-lowerboundhilbert} and \ref{prop-dim2} are contained in Section~\ref{proof lower bound} while the proof of Theorems~\ref{main thm} and \ref{main thm2}, and Proposition~\ref{main thm3} will be given in Section~\ref{proof plank}.

\section{Main results}\label{results}

We begin this section with some notation. Given a Banach space $X$, $B_X$ denotes the unit ball of $X$.  Recall that a function $P:X\rightarrow \zK$ is a continuous $k-$homogeneous polynomial if there is a continuous $k-$linear function $T:X\times\cdots\times X\rightarrow \zK$ for which $P(\textbf{z})=T(\textbf{z},\ldots,\textbf{z})$. A function $Q:X\rightarrow \zK$ is a continuous polynomial of degree $k$ if $Q=\sum_{l=0}^k Q_l$, with $Q_0$ a constant, $Q_l$  an $l-$homogeneous polynomial ($1\leq l \leq k$) and $Q_k \neq 0$. The norm of a polynomial $Q$ is defined as $$\Vert Q\Vert = \sup_{\textbf{z}\in B_X} \vert Q(\textbf{z})\vert.$$

\subsection{Lower bounds for the product of polynomials}

Our main result concerning lower bounds for the product of polynomials on finite dimensional spaces is the following.

\begin{thm}\label{prop-cualquierespacio} Let $X$ be a $d-$dimensional Banach space over  $\zK=\mathbb R$ or $\mathbb C$ and $P_1,\ldots,P_n$ scalar polynomials of degrees $k_1,\ldots,k_n$ over $X$. Then
\begin{equation}\label{eq-cualquierespacio}
 \Vert P_1 \Vert \cdots \Vert P_n \Vert \leq \frac{(C_{\zK}4ed)^{\sum_{i=1}^n k_i}}{2^{\frac{n}{C_{\zK}}}} \Vert P_1 \cdots P_n\Vert ,
\end{equation}
where $C_{\zR}=1$ and $C_{\zC}=2$.
\end{thm}

Let us compare this result with the results on this problem mentioned in the Introduction. The inequality \eqref{eq-cualquierespacio} is better than \eqref{eq-BST} for finite dimensional spaces provided that the number of polynomials is large enough. The same holds true when we apply \eqref{eq-cualquierespacio} to homogeneous polynomials on finite dimensional $L_p$ spaces, with $1\leq p \leq 2$, and compare  it with \eqref{eq-P} and \eqref{eq-CPR}.

Comparing \eqref{eq-cualquierespacio} and \eqref{eq-MNPP} for finite dimensional real Hilbert spaces it is not as straightforward. In this case, which bound is better depends on the particular setting. For example, when we consider $n$ homogeneous polynomials of the same degree $k$ on a $d$ dimensional real Hilbert space, for large values of $k$ we see that \eqref{eq-cualquierespacio} is better than  \eqref{eq-MNPP} if $n > 8e^2d^2$, and that \eqref{eq-MNPP} is better than \eqref{eq-cualquierespacio} if $n< 8e^2d^2$. On the other hand, if we fix $k$, \eqref{eq-cualquierespacio} is better than  \eqref{eq-MNPP} for $n$ large enough.

\smallskip
Theorem~\ref{prop-cualquierespacio} is, in some sense we now explain,  sharp as a general result. A look at Theorem~\ref{prop-cualquierespacio} suggests an extension of the problem of finding the linear polarization constant of a space $X$ (see \cite{PR,RS}). We first define $\MMM_n(X)$ as the optimal constants such that, for any set of continuous scalar polynomials $P_1, \ldots, P_n$ over $X$, of degrees $k_1,\ldots, k_n$, the following inequality holds
\begin{equation}\nonumber
 \Vert P_1 \Vert \cdots \Vert P_n \Vert \leq \MMM_n(X)^{\sum_{i=1}^n k_i} \Vert P_1 \cdots P_n\Vert.
\end{equation}
Then we set
\begin{equation}\label{def M}
\MMM(X)= \overline{\displaystyle\lim_{{n\rightarrow \infty}}}\,\,\MMM_n(X).
\end{equation}
As a consequence of Theorem~\ref{prop-cualquierespacio} we have
$$\MMM(X) \leq C_{\zK}4ed.$$ Let us see that taking $X=\ell_1^d$ the linear growth in $d$ can be attained, which shows the sharpness of Theorem~\ref{prop-cualquierespacio}
. In this case, for $n > d$, we define for $i=1,\dots,n$ the $k$-homogeneous polynomial $P_i$ on $\ell_1^d(\zK)$ by
\[
	P_i(z_1,\ldots,z_d)= \left\{	 \begin{array}{lcl}
                  z_i^{k(n-d+1)}  & \mbox{ if } & i< d \\																
                    &             &     \\
                  z_d^{k} & \mbox{ if } & i\geq d
\end{array}
\right.
\]
This is a set of $n$ polynomials and, using Lagrange multipliers (see Lemma 2.3 from \cite{CPR}), it is easy to see the following:
\begin{eqnarray}
\Vert P_1 \cdots P_n\Vert & = & \frac{1}{d^{dk(n-d+1)}}\Vert P_1 \Vert  \cdots \Vert P_{d-1}\Vert \Vert P_{d}   \cdots  P_{n}\Vert \nonumber \\
&=& \frac{1}{d^{\sum_{i=1}^n\deg(P_i)}}\Vert P_1 \Vert  \cdots \Vert P_{d-1}\Vert \Vert P_{d} \Vert  \cdots \Vert P_{n}\Vert. \nonumber \
\end{eqnarray}
Therefore, $\MMM_n(\ell_1^d(\zK)) \geq d$, and thus $\MMM(\ell_1^d(\zK)) \geq d$. Then, we conclude that $\MMM(\ell_1^d(\zK)) $ increases at the same rate as $d$. Note that $k$ was arbitrarily, so we cannot improve this growth rate by fixing the degrees of the polynomials.

\medskip

For finite dimensional Hilbert spaces, we have the following result, which gives better bounds than those of Theorem~\ref{prop-cualquierespacio}.

\begin{prop}\label{prop-lowerboundhilbert} Let $\mathcal H$ be a $d$ dimensional (real or complex) Hilbert space and $P_1,\ldots,P_n$ scalar homogeneous polynomials of degrees $k_1,\ldots,k_n$ over $\mathcal H$. Then
\begin{equation}\label{eq-cualquierhilbert}
\Vert P_1 \Vert \cdots \Vert P_n \Vert \leq\left(\frac{e^{H_{dC_{\zK}}}}{4}\right)^{\sum_{i=1}^n k_i}\Vert P_1 \cdots P_n\Vert ,
\end{equation}
where $C_{\zR}=1$, $C_{\zC}=2$ and $H_d$ stands for the $d$th harmonic number $\sum_{k=1}^d \frac 1 k$.
\end{prop}

To see that this results improves Theorem~\ref{prop-cualquierespacio} applied to Hilbert spaces it is enough to see that
$$\frac{ e^{H_{dC_{\zK}}} /4}{C_{\zK} 2ed}< 1 .$$
This follows from the fact that the sequence $H_l -\log(l)$ decreases, hence
$$
\frac{ e^{H_{dC_{\zK}}} /4}{C_{\zK} 2ed} \leq  \frac{ e^{H_1} /4}{2e} = \frac{1}{8}.$$

\medskip

Finally, using that every homogeneous polynomial $P:\zC^2\rightarrow \zC$ is the product of linear functions (see, for example, \cite[Lemma 3.3.6]{Ro2}), we obtain the following estimate for $2$-dimensional complex Banach spaces.

\begin{prop}\label{prop-dim2} Let $X$ be a complex $2$-dimensional Banach space and $P_1,\ldots,P_n$ scalar homogeneous polynomials of degrees $k_1,\ldots,k_n$ over $X$. Then
\begin{equation}\nonumber
\Vert P_1 \Vert \cdots \Vert P_n \Vert \leq   \ccc_{\kkk}(X)\Vert P_1 \cdots P_n\Vert ,
\end{equation}
where $\kkk=\sum_{j=1}^n k_i$ and $\ccc_{\kkk}(X)$ is the $\kkk-$th linear polarization constant of $X$.
\end{prop}

This last result can be used to relate the linear polarization constant $\ccc(X)$ with the constant $\MMM(X)$ defined in  \eqref{def M}.

\begin{cor} Let $X$ be a complex $2$-dimensional Banach space, then
$$\MMM(X)=\ccc(X).$$
\end{cor}
\begin{proof}
The inequality
$$\MMM(X) \geq \ccc(X)$$
is immediate ---and holds for any Banach space---  since $\MMM_n(X) \geq \ccc_n(X)^{\frac 1 n}.$

For the other inequality, by Proposition \ref{prop-dim2}, we have
\begin{eqnarray}
\MMM_n(X) &\leq & \sup \{ \ccc_{\kkk}(X)^{\frac{1}{\kkk}}:k_1,\ldots,k_n\in\zN, \kkk =\sum_{j=1}^n k_i\} \nonumber \\
&= &\sup \{ \ccc_{l}(X)^{\frac{1}{l}}:l\geq n\}, \nonumber \
\end{eqnarray}
therefore
\begin{eqnarray}
\MMM(X) &\leq & \lim_{l\to \infty} \sup \{ \ccc_{l}(X)^{\frac{1}{l}}:l\geq n\} =\ccc(X). \nonumber \qedhere
\end{eqnarray}
\end{proof}

\subsection{The polynomial plank problem}

Our first main plank type result, and our most general one,  is the following.

\begin{thm}\label{main thm}\label{section planks principales}
Let $X$ be a complex Banach space and $P_1,\ldots,P_n:X \rightarrow \zC$ be norm one polynomials of degrees $k_1,\dots,k_n$. Given $a_1,\ldots,a_n \in \zR_{\geq 0}$ satisfying \mbox{$\sum_{i=1}^n a_i < \frac{1}{n^{n-1}}$},  there is $\textbf{z}_0 \in B_{X}$ such that
$$|P_i(\textbf{z}_0)| \geq a_i^{k_i} \mbox{ for } i=1,\ldots,n.$$
Moreover, if $X$ is finite dimensional, this also holds for $\sum_{i=1}^n a_i = \frac{1}{n^{n-1}}$.
\end{thm}

The proof of this theorem will make use of the constant provided in \eqref{eq-BST}.
Although the constant~\eqref{eq-BST} is optimal for general complex Banach spaces, in some cases better constants have been obtained. As we have already mentioned, if $X$ is the complex Banach space $L_p(\mu)$ or $\mathcal S_p$, with $1\le p\le 2$, we have \eqref{eq-CPR}. Then, using \eqref{eq-CPR} instead of \eqref{eq-BST}, we obtain the following plank result.

\begin{thm}\label{main thm2}
Let $X$ be the complex Banach space $L_p(\mu)$  or $\mathcal S_p$, with $1\le p\le 2$, and $P_1,\ldots,P_n:X \rightarrow \zC$ be norm one homogeneous polynomials of degrees $k_1,\dots,k_n$. Given $a_1,\ldots,a_n \in \zR_{\geq 0}$ satisfying $\sum_{i=1}^n a_i^p < \frac{1}{n^{n-1}}$,  there is $\textbf{z}_0 \in B_{X}$ such that
$$|P_i(\textbf{z}_0)| \geq a_i^{k_i} \mbox{ for } i=1,\ldots,n.$$
Moreover, if $X$ is finite dimensional, this also holds for $\sum_{i=1}^n a_i^p = \frac{1}{n^{n-1}}$.
\end{thm}

As pointed out before, the constants provided in \eqref{eq-BST} and \eqref{eq-CPR} are optimal when the dimension of the underlying spaces are at least $n$ (the number of polynomials). Then, a natural next step is to use the sharper inequalities obtained for finite dimensional spaces in Theorem~\ref{prop-cualquierespacio}. To do this, we need the following.

\begin{prop}\label{main thm3}
Let $X$ be a finite dimensional Banach space, $n$ a natural number and suppose we have a positive constant $K<\frac{1}{\sqrt[n]{ne^2}}$ such that for any set  $P_1,\ldots,P_n:X \rightarrow \zC$  of norm one  polynomials we have
$$
\|P_1\cdots P_n\| \ge K^{\sum_{i=1}^n k_i} \, \|P_1\| \cdots \|P_n\|,
$$
where $k_1,\dots,k_n$ are the degrees of the polynomials. Then, given $a_1,\ldots,a_n \in \zR_{\geq 0}$, with $\sum_{i=1}^n a_i \leq nK^n$, there is $\textbf{z}_0 \in B_{X}$ such that
$$|P_i(\textbf{z}_0)| \geq a_i^{k_i}  \mbox{ for } i=1,\ldots,n.$$
\end{prop}

Then, combining Proposition~\ref{main thm3} and Theorem~\ref{prop-cualquierespacio}, we obtain the following plank type result for polynomials on finite dimensional spaces.

\begin{prop}Let $X$ be a $d$-dimensional Banach space over $\zK$, $C_\zK$ as in Theorem~\ref{prop-cualquierespacio} and $P_1,\ldots,P_n:X \rightarrow \zK$ a set of norm one polynomials of degrees $k_1,\dots,k_n$. Given  $a_1,\ldots,a_n \in \zR_{\geq 0}$ satisfying $\sum_{i=1}^n a_i \leq n\left(\frac{1}{C_\zK 4ed}\right)^n$, there is $\textbf{z}_0 \in B_{X}$ such that
$$|P_i(\textbf{z}_0)| \geq a_i^{k_i}  \mbox{ for } i=1,\ldots,n.$$

\end{prop}

\begin{rem}\label{remark-comparacion} It is natural to compare the plank type results described in this article to previous work. First, it is easy to see that for linear functions (i.e. homogeneous polynomials of degree one), we are far from recovering the optimal results of K.~Ball on the plank problem. On the other hand, the value of our results relies on the generality in which they can be stated. They can be applied for polynomials of arbitrary (and different) degrees, and a large range of positive numbers $a_1,\ldots, a_n$. Moreover, most of them also work for non homogeneous polynomials. In this way, we extend, and sometimes improve, previous work in the subject. For example, Theorem 5 of \cite{KK} can be recovered from Theorem~\ref{main thm} and Theorem~\ref{main thm2} as a particular case, taking polynomials of the same degrees, all the scalars with the same value, etc.
\end{rem}

\section{The Proofs of the lower bounds for products of polynomials}\label{proof lower bound}

\subsection{Proof of Theorem \ref{prop-cualquierespacio}}

In order to prove Theorem~\ref{prop-cualquierespacio} we will use Remez type inequalities for polynomials in several variables. The objective of Remez type inequalities is to give bounds for classes of functions over some fixed set, given that the modulus of the functions is bounded on some subset of prescribed measure. For example, the original inequality of Remez states the following.

\begin{quote}

Take $a>0$ and a polynomial $P:[-1,1+a]\rightarrow \zR$ of degree $k$ such that
$$\displaystyle\sup_{t\in V} |P(t)| \leq 1$$
for some measurable subset $V\subseteq [-1,1+a]$, with $|V|\geq 2$. Then
$$\displaystyle\sup_{t\in [-1,1+a]} |P(t)| \leq \displaystyle\sup_{t\in [-1,1+a]} |T_k(t)|, $$
where $T_k$ stands for the Chebyshev polynomial of degree $k$.
\end{quote}

This inequality, combined with some properties of the Chebyshev polynomials, produces the following corollary, which most applications of Remez inequality use.

\begin{cor} Let  $P:\zR\rightarrow \zR$ be a polynomial of degree $k$, $I\subset \zR$ be an interval and $V\subseteq I$ an arbitrary measurable set, then
\begin{equation}\label{remezoriginal}
\displaystyle\sup_{t\in I} |P(t)| \leq \left( \frac {4 |I|}{|V|} \right)^k \displaystyle\sup_{t\in V} |P(t)|.
\end{equation}
\end{cor}

We are interested in inequalities similar to \eqref{remezoriginal}, but for polynomials on several variables. Y.~Brudnyi and M.~Ganzburg studied Remez type inequalities for polynomials on several variables in \cite{BG}. As the original result of Remez, they stated their main result in terms of the Chebyshev polynomials.
\begin{thm}\label{teo brudnyi} 	Let $X$ be a $d$-dimensional real space, $\lambda$ a positive number and $P:X\rightarrow\zR$ a polynomial   of degree $k$ such that
$$\displaystyle\sup_{t\in V} |P(t)| \leq 1$$
for some measurable subset $V\subseteq B_X$, with $\mu(V)\geq \lambda$, where $\mu$ is the normalized Lebesgue measure over $B_X$. Then
$$\Vert P\Vert \leq T_k\left( \frac{ 1+ (1- \lambda)^{\frac 1 d} }{ 1- (1-\lambda)^{\frac 1 d} }\right).$$
\end{thm}

Just as in the applications of the original Remez inequality, rather than using Theorem~\ref{teo brudnyi}, we will use next proposition (see inequality (8) from \cite{BG}), which is a corollary from the main result of \cite{BG}.

\begin{prop} Let $X$ be a $d$-dimensional real space and $P:X\rightarrow\zR$ a polynomial of degree $k$. Given any Lebesgue measurable subset $V\subseteq B_X$, we have
$$\displaystyle\sup_{\textbf{z}\in B_X} |P(\textbf{z})| \leq \frac{1}{2} \left( \frac{4d}{\mu(V)}\right)^k \displaystyle\sup_{\textbf{z}\in V} |P(\textbf{z})|,$$
where $\mu$ is the normalized Lebesgue measure over $B_X$.
\end{prop}
As an immediate consequence of this result, we have the following inequality (see inequality (14) from \cite{BG}). If $P$ is a norm one polynomial of degree $k$ over a finite $d$-dimensional Banach space $X$, then
\begin{equation}\label{remez}
\mu(\{\textbf{z}\in B_X: |P(\textbf{z})| \leq t\}) \leq 4d \left(\frac{t}{2}\right)^{\frac 1 k},
\end{equation}
for any $0<t<1$.

Finally, we will need to use the following lemma.
\begin{lem}\label{lem8} Let $P:X\rightarrow \zR$ be a norm one polynomial of degree $k$. Then
$$\int_0^{+\infty} \mu(\{\textbf{z}\in B_X: |P(\textbf{z})| \leq e^{-t}\})\,\,\, dt\leq  -\ln\left(\frac{2}{(4d)^{k}}\right) + k,$$
where $\mu$ is the normalized Lebesgue measure over $B_X$.
\end{lem}
\begin{proof}To simplify notation let us writte $$V_t=\{\textbf{z}\in B_X: |P(\textbf{z})| \leq e^{-t}\}.$$ Then, using the Remez type inequality~(\ref{remez}), we have
\begin{eqnarray*}
\int_0^{+\infty} \mu(V_t)\,\,\,dt&=& \int_0^{-\ln \left(\frac{2}{(4d)^{k}}\right)} \mu(V_t)\,\,\,dt +\int_{-\ln\left(\frac{2}{(4d)^{k}}\right)}^{+\infty} \mu(V_t)\,\,\,dt \nonumber \\
&\leq& \int_0^{-\ln\left(\frac{2}{(4d)^{k}}\right)} 1 \,\,\,dt +\int_{-\ln\left(\frac{2}{(4d)^{k}}\right)}^{+\infty} 4d \frac{e^{\frac{-t}{k}}}{2^{\frac 1 k}} \,\,\,dt\nonumber \\
&=& -\ln\left(\frac{2}{(4d)^{k}}\right) + \frac{4d}{2^{\frac{1}{k}}}    (-k)e^{\frac{-t}{k}} \Bigg|_{-\ln\left(\frac{2}{(4d)^{k}}\right)}^{+\infty} \nonumber \\
&=& -\ln\left(\frac{2}{(4d)^{k}}\right) + k.  \qedhere \nonumber
\end{eqnarray*}
\end{proof}

\begin{proof}[Proof of Theorem \ref{prop-cualquierespacio}]

Given $P_1,\dots,P_n:X\rightarrow \zK$  polynomials of degree $k_1,\dots,k_n$, we have to prove that
\begin{equation}\nonumber
\|P_1\cdots P_n\| \ge \frac{2^{\frac{n}{C_{\zK}}}}{(C_{\zK}4ed)^{\sum_{i=1}^n k_i}} \, \|P_1\| \cdots \|P_n\|.
\end{equation}
We start with the real case. We may assume all the polynomials have norm one. Using Lemma~\ref{lem8} we have:
\begin{eqnarray}
\ln(\|P_1\cdots P_n\|)&=&\displaystyle\sup_{\textbf{z}\in B_X}  \ln\left(\prod_{i=1}^n |P_i(\textbf{z})|\right) \nonumber \\
&\geq&  \int_{B_X} \ln\left(\prod_{i=1}^n |P_i(\textbf{z})|\right) \,\,\,d\mu(\textbf{z})  \nonumber \\
&=&\sum_{i=1}^n  \int_{B_X} \ln |P_i(\textbf{z})|\,\,\, d\mu(\textbf{z}) \nonumber \\
&=&-\sum_{i=1}^n  \int_{B_X} -\ln |P_i(\textbf{z})|\,\,\, d\mu(\textbf{z}) \nonumber \\
&=&-\sum_{i=1}^n  \int_0^{+\infty} \mu(\{\textbf{z}\in B_X: |P_i(\textbf{z})| \leq e^{-t}\}) \,\,\,dt \nonumber \\
&\geq&\sum_{i=1}^n  \ln\left(\frac{2}{(4d)^{k_i}}\right) - k_i. \nonumber \
\end{eqnarray}
Therefore
\begin{eqnarray*}
\|P_1\cdots P_n\|&\geq& \exp\left\{ \sum_{i=1}^n  \ln\left(\frac{2}{(4d)^{k_i}}\right) - k_i \right\} \\
&=&\prod_{i=1}^n \frac{2}{(4d)^{k_i}}\frac{1}{e^{k_i}} \\
&=&\frac{2^n}{(4de)^{\sum_{i=1}^n k_i}}, \
\end{eqnarray*}
as desired.

To prove the complex case we will use the real case. Let $X$ be a $d-$dimensional complex Banach space and $P_1,\dots,P_n:X\rightarrow \zC$ polynomials of degree $k_1,\dots,k_n$. Take $Y$ the $2d-$dimensional real Banach space obtained from thinking $X$ as a real space, and consider the polynomials $Q_1,\dots,Q_n:Y\rightarrow \zR$, of degrees $2k_1,\dots,2k_n$, defined as
$$Q_i(\textbf{z})=|P_i(\textbf{z})|^2, \ \quad  i=1,\dots,n.$$
Applying inequality \eqref{eq-cualquierespacio} for polynomials on a real Banach space to these polynomials we obtain
\begin{eqnarray}
\|P_1\|^2 \cdots \|P_n\|^2 \frac{2^n}{(8de)^{\sum_{i=1}^n 2k_i}} &=& \|Q_1\| \cdots \|Q_n\| \frac{2^n}{(8de)^{\sum_{i=1}^n 2k_i}} \nonumber \\
&\leq&  \|Q_1 \cdots Q_n\|=   \|P_1 \cdots P_n\|^2,\nonumber \
\end{eqnarray}
which ends the proof.
\end{proof}

\subsection{Proof of Proposition~\ref{prop-lowerboundhilbert}}

A cornerstone on the proof of Theorem~\ref{prop-cualquierespacio} was the use of the Remez type ine\-qua\-li\-ty \eqref{remez} to obtain Lemma~\ref{lem8}.
But when we restrict ourselves to homogeneous polynomials over Hilbert spaces we can prove the following sharper lemma (see \cite{Ro2}, Lemma~3.3.4).

\begin{lem}\label{lem8homog} Let $P:\ell_2^d(\zR) \rightarrow \zR$ be a norm one homogeneous polynomial of degree~$k$. Then
$$\int_0^{+\infty} \mu(\{\textbf{z}\in B: |P(\textbf{z})| \leq e^{-t}\})\leq  k(\ln\left(4\right) + H_d),$$
where $H_d$ stands for the $d$th harmonic number.
\end{lem}

Then,  if we replace Lemma~\ref{lem8} with Lemma~\ref{lem8homog} in the proof of Theorem \ref{prop-cualquierespacio}, we obtain the proof of Proposition \ref{prop-lowerboundhilbert}.

\subsection{Proof of Proposition~\ref{prop-dim2}}

We may assume $P_1,\dots,P_n$ are norm one polynomials. By Lemma 3.3.6  from \cite{Ro2}, we know that
$$P_i = L_i\varphi_{i,1}\cdots \varphi_{i,k_i}\mbox{ for } i=1,\ldots,n,$$
where $\varphi_{i,j}$ are norm one linear functions and $$L_i= \frac{1}{\| \varphi_{i,1}\cdots \varphi_{i,k_i}\|}\ge 1.$$ Then, by definition of the $\kkk-$th linear polarization constan, we have

\begin{eqnarray}
\|P_1\cdots P_n\| &=&\left\Vert\prod_{i=1}^n \left( L_i \prod_{j=1}^{k_i} \varphi_{i,j}\right)\right\Vert  \nonumber \\
&=&  \left(\prod_{i=1}^n L_i\right) \left\Vert \prod_{i=1}^n \prod_{j=1}^{k_i} \varphi_{i,j}\right\Vert \nonumber \\
&\ge&  \left\Vert \prod_{i=1}^n \prod_{j=1}^{k_i} \varphi_{i,j} \right\Vert \nonumber \\
&\ge& \frac{1}{\ccc_{\kkk}(X)} \nonumber\
\end{eqnarray}
which ends the proof. \hfill \qed

\section{Proof of the plank type results}\label{proof plank}
\subsection{Proof of Theorems \ref{main thm} and \ref{main thm2}}
For the proof of Theorems~\ref{main thm} and~\ref{main thm2} we need the following two technical lemmas.

\begin{lem}\label{lem5}
Given $n$ positive integers $k_1,\ldots,k_n$, the set \begin{equation}\label{conjuntodenso}\left\{ \frac{1}{\sum_{i=1}^n k_ir_i}(k_1r_1,\ldots,k_nr_n): r_1,\dots,r_n\in \zN \right\}\end{equation} is dense in $\{x \in \zR^n : \sum_{i=1}^n x_i = 1, x_i\geq 0 \}$.
\end{lem}
\begin{proof}
Take a set of rational numbers $t_1,\ldots,t_n\in \zQ,$ with $t_i>0$ ($i=1,\dots,n$) and \mbox{$\sum_{i=1}^n t_i =1$}. Write $t_i=\frac{q_i}{p}$ with $q_1,\ldots,q_n\in \zN$ and $p=\sum_{i=1}^n q_i$. Let $M=\prod_{i=1}^nk_i$ and take $r_i$ such that $k_ir_i=q_iM$. Then
$$\frac{1}{\sum_{i=1}^n k_ir_i}(k_1r_1,\ldots,k_nr_n)=\frac{1}{\sum_{i=1}^n q_iM}(q_1M,\ldots,q_nM)=\left(\frac{q_1}{p},\ldots,\frac{q_n}{p}\right).$$ Now, the density of rational numbers gives the desired result.
\end{proof}

\begin{lem}\label{lem4}
Given $b_1,\ldots,b_n \in \zR_{\geq 0}$, with $\sum_{i=1}^n b_i = \frac{1}{n^{n-1}}$, there is an element $(t_1,\ldots,t_n)\in \zR_{>0}^n$  such that
$$ \sum_{i=1}^n t_i = 1\quad\text{and}\quad
t_1^{t_1}\cdots t_n^{t_n} \geq b_i^{t_i} \mbox{ for } i=1,\ldots,n.$$
\end{lem}
\begin{proof}
For $n=1$ we can take $t_1=1$ and for $n=2$ we can take $t_1=t_2=\frac 1 2$. Let us assume $n\geq 3$. For each $i$ let us define $\tilde{t}_i=-\frac{c}{\ln(b_i)}$, with $c>0$ such that
\begin{equation}\label{letra c} \prod_{j=1}^n \tilde{t}_j^{\tilde{t}_j}=e^{-c}.
\end{equation}
If we show $\sum_{j=1}^n \tilde{t}_j \leq 1$, and define $t_i=\dfrac{1}{\sum_{j=1}^n \tilde{t}_j} t_i$, then we have
$$\prod_{j=1}^{n} t_j^{t_j}= \frac{1}{\sum_{j=1}^n \tilde{t}_j} \left(\prod_{j=1}^{n} \tilde{t}_j^{\tilde{t}_j}\right)^{\frac{1}{\sum_{j=1}^n \tilde{t}_j}}\geq \prod_{j=1}^n \tilde{t}_j^{\tilde{t}_j}=e^{-c}=b_i^{\tilde{t}_i}\geq b_i^{t_i}.$$
Let us see this. Define $a_i=\frac{-1}{\ln(b_i)}$. Since each $b_i\leq \frac{1}{n^{n-1}}<\frac{1}{e^2}$ and the map $x\mapsto \frac{-1}{\ln(x)}$ is concave on $\left(0,\frac{1}{e^2}\right)$, by Jensen's inequality, we have
\begin{equation}\label{cota para a}\sum_{i=1}^n a_i \leq \frac{1}{\ln(n)}.
\end{equation}
From \eqref{letra c} we have
$$\ln(c)=\frac{-1-\sum_{i=1}^n a_i\ln(a_i)}{\sum_{i=1}^n a_i}.$$
Now, using that $x\mapsto x\ln(x)$ is convex on $\zR_{>0}$, Jensen's inequality and \eqref{cota para a} we obtain
\begin{eqnarray}
\frac{-1-\sum_{i=1}^n a_i\ln(a_i)}{\sum_{i=1}^n a_i} & \leq & \frac{-1 - \sum_{i=1}^n a_i(\ln(\sum_{i=1}^n a_i) - \ln(n))}{\sum_{i=1}^n a_i} \nonumber \\
&=&-\frac{1}{\sum_{i=1}^n a_i}-\ln\left(\sum_{i=1}^n a_i\right)+\ln(n) \nonumber \\
&\leq & -\frac{1}{1/\ln(n)}-\ln(1/\ln(n))+\ln(n) \label{increasing} \\
&=& \ln(\ln(n)), \nonumber \
\end{eqnarray}
where in \eqref{increasing} we use that the map $x\mapsto -\frac 1 x -\ln(x)$ is increasing on $(0,1)$. Then $c\leq \ln(n)$ and therefore
$$\sum_{j=1}^n \tilde{t}_j =\sum_{j=1}^n c a_j  \leq \ln(n) \sum_{j=1}^n  a_j \leq 1$$
which ends the proof.
\end{proof}

Note that the inequalities in Theorem~\ref{main thm} and \eqref{eq-BST} are exactly the ones in  Theorem~\ref{main thm2} and \eqref{eq-CPR} when we take $p=1$. With this in mind, it is easy to see that, if we put $p=1$ in the proof of Theorem~\ref{main thm2} below, we obtain the proof of  Theorem~\ref{main thm}.

\begin{proof}[Proof of  Theorem  \ref{main thm2}]
Choose $b_i>a_i^p$, $i=1,\dots,n$, such that $\sum_{i=1}^n b_i =\frac{1}{n^{n-1}}$.
By Lemma~\ref{lem4}, we can take  an element $(t_1,\ldots,t_n)\in \zR_{>0}^n $ with $\sum_{i=1}^n t_i=1$ and
$$t_1^{t_1}\cdots t_n^{t_n} \geq b_i^{t_i} \, \mbox{ for } \, i=1,\ldots,n.$$

We claim that there is $\delta>0$  such that for any positive integer $N$, we can choose $r=(r_1,\ldots,r_n)\in \zN^n$ so that, if we call $s_i = \frac{k_ir_i}{\sum_{j=1}^n k_jr_j}$ , we have
$$s_i\geq \delta\quad \text{and}\quad s_1^{s_1}\cdots s_n^{s_n} \geq b_i^{s_i} \left(1-\frac{1}{N}\right) \quad \mbox{ for } i=1,\ldots,n.$$
In the notation we omit the dependence on $N$ for simplicity. Indeed, using the convention that the function $t\mapsto t^t$ is one at $t=0$, we have the finite family of continuous functions $x\mapsto x_1^{x_1}\cdots x_n^{x_n}$ and $x\mapsto b_i^{x_i}$, $i=1,\dots,n$, defined on the compact set $\{x \in \zR^n : \sum_{i=1}^n x_i = 1, x_i\geq 0 \}$. Then, the density of the set \eqref{conjuntodenso} in this compact set is just what we need for finding $r$ such that $s_1^{s_1}\cdots s_n^{s_n} \geq b_i^{s_i} \left(1-\frac{1}{N}\right)$. Since we have to take $s_i$ close enough to $t_i$ we also may assume $s_i \geq \frac{t_i}{2}$. Then, we define $\delta :=\min\left\{\frac{t_i}{2}:i=1,\ldots,n\right\}$.

On the other hand, by \eqref{eq-CPR}, we have
$$\left\Vert \prod_{i=1}^n P_i^{r_i} \right\Vert \geq \left(\frac{\prod_{i=1}^n (k_ir_i)^{k_ir_i} }{(\sum_{i=1}^n k_ir_i )^{\sum_{i=1}^n k_ir_i}}\right)^{\frac 1 p}.$$
So, for all $N \in \zN$, we can take $\textbf{z}_N \in B_{X}$ such that
$$| \big(P_1(\textbf{z}_N)\big)^{r_1}\cdots \big(P_n(\textbf{z}_N)\big)^{r_n}|\geq \left(\frac{\prod_{i=1}^n (k_ir_i)^{k_ir_i} }{(\sum_{i=1}^n k_ir_i )^{\sum_{i=1}^n k_ir_i}}\right)^{\frac 1 p }\left( 1 - \frac 1 N \right).$$
Since each polynomial has norm one, this gives for each $i=1,\dots,n$:
$$|\big(P_i(\textbf{z}_N)\big)^{r_i}| \geq  \left(\frac{\prod_{j=1}^n (k_jr_j)^{k_jr_j} }{(\sum_{j=1}^n k_jr_j )^{\sum_{j=1}^n k_jr_j}}\right)^{\frac 1 p }\left( 1 - \frac 1 N \right).$$
Therefore,
\begin{eqnarray}
|P_i(\textbf{z}_N)| &\geq& \left(\frac{\prod_{j=1}^n (k_jr_j)^{k_jr_j} }{(\sum_{j=1}^n k_jr_j )^{\sum_{j=1}^n k_jr_j}}\right)^{\frac{k_i}{k_ir_ip}} \left( 1 - \frac 1 N \right)^{\frac{1}{r_i}}\nonumber \\
&=& \left(s_1^{s_1}\cdots s_n^{s_n}\right)^{\frac{ k_i }{s_i p}} \left( 1 - \frac 1 N \right)^{\frac{1}{r_i}} \nonumber \\
&\geq& \left(b_i^{s_i} \left(1-\frac{1}{N}\right) \right)^{\frac{ k_i }{s_i p}} \left( 1 - \frac 1 N \right) \nonumber \\
&=& b_i^{\frac{k_i }{p}} \left(1-\frac{1}{N}\right)^{\frac{ k_i }{s_i p} +1} \nonumber \\
&\geq& b_i^{\frac{k_i }{p}} \left(1-\frac{1}{N}\right)^{\frac{ k_i }{\delta p}+1}. \label{ultimopaso} \
\end{eqnarray}

But recall that $b_i>a_i^p$, $i=1,\dots,n$. Since $\delta$  does not depend on $N$, we can take $N$ large enough such that
$$b_i^{\frac{k_i }{p}} \left(1-\frac{1}{N}\right)^{\frac{ k_i }{\delta p}+1} \geq a_i^{k_i }, $$
which ends the proof of the general case.

For the finite dimensional case, we need to deal with the case $\sum_{i=1}^n a_i^p = \frac{1}{n^{n-1}}$. For this, we take $b_i=a_i^p$ and proceed as in the proof of the general case up to \eqref{ultimopaso}. We can take, by the finite dimension of our space, a limit point $\textbf{z}_0\in B_{X}$  of the sequence $\{\textbf{z}_N\}_{N \in \zN}$. Then, by continuity, we have
$$|P_i(\textbf{z}_0)|\geq b_i^{\frac{k_i }{p}}=a_i^{k_i },$$
as desired.
\end{proof}

\subsection{Proof of Proposition \ref{main thm3}}

The proof of Proposition~\ref{main thm3} is analogous to the proof of Theorems~\ref{main thm} and \ref{main thm2}, replacing Lemma~\ref{lem4} with the following lemma.

\begin{lem}\label{lem7}
Let $n$ be a natural number and $K\in \left(0,\frac{1}{\sqrt[n]{ne^2}}\right]$. Given non negative numbers $b_1,\ldots,b_n$, with $\sum_{i=1}^n b_i = nK^n$, there is an element $(t_1,\ldots,t_n)\in \zR_{>0}^n$  such that
$$ \sum_{i=1}^n t_i = 1\quad\text{and}\quad
K^{\frac{1}{t_i}} \geq b_i \mbox{ for } i=1,\ldots,n.$$
\end{lem}

\begin{proof}
Let us first assume $b_1,\ldots,b_n$ are strictly positive and define $s_i:= \frac{\ln(K)}{\ln(b_i)}$, or equi\-va\-lently  $b_i=K^{\frac{1}{s_i}}$. If we show that $\sum_{j=1}^ns_j\leq 1$, then  we can take $t_i\geq s_i$ such that $\sum_{j=1}^nt_j =1$ and, since $x \mapsto K^{\frac{1}{s}}$ is increasing, we have
$$ K^{\frac{1}{t_i}} \geq K^{\frac{1}{s_i}} = b_i.$$

Let us see then, that $\sum_{j=1}^ns_j\leq 1$. The condition  $\sum_{j=1}^n b_j = nK^n$ implies $b_i\leq\frac{1}{e^2}$ for each $i=1,\ldots,n$. Since the function $x\mapsto\frac{\ln(K)}{\ln(x)}$ is concave on $\left[0,\frac{1}{e^2}\right]$, using Jensen's inequality we have
$$\sum_{j=1}^n s_j =\sum_{j=1}^n \frac{\ln(K)}{\ln(b_j)} \leq n \frac{\ln(K)}{\ln\left(\frac{\sum_{j=1}^n b_j}{n}\right)}=1.$$

If $b_{i_0}=0$ for some $i_0$, we define $s_i=0$ whenever $b_i=0$ and $s_i:= \frac{\ln(K)}{\ln(b_i)}$ otherwise.
Since in this case we do not have $b_1=b_2=\ldots =b_n$, proceeding as in the previous case we obtain
$$\sum_{j=1}^n s_j<1.$$
This allow us to take each $t_i$ strictly greater than $s_i$ (and, in particular, strictly positive as desired), satisfying $\sum_{j=1}^nt_j =1$. We go on as above to obtain the result.
\end{proof}



\section*{Acknowledgements}
This project was supported in part by UBACyT 20020130100474BA, PIP  11220130100329CO (CONICET), PICT 2011-1456 and PICT-2015-2299.

\end{document}